\documentclass[reqno]{amsart} 
\usepackage{amsfonts, amsmath, amsthm, amssymb, latexsym, graphicx, geometry, xcolor, colortbl,  mathtools}
\usepackage{color}

\def\b{\mathbb }
\def\cal{\mathcal }

\theoremstyle{plain}
\newtheorem{theorem}{Theorem}[section]
\newtheorem{corollary}[theorem]{Corollary}
\newtheorem{lemma}[theorem]{Lemma}
\newtheorem{proposition}[theorem]{Proposition}
\newtheorem{def-theorem}[theorem]{Definition and Theorem}
\theoremstyle{definition}
\newtheorem{definition}[theorem]{Definition}

\newtheorem{example}[theorem]{Example}

\numberwithin{equation}{section}

\title[A central limit theorem for Bessel  and Dunkl processes with drift]{A central limit theorem for Bessel  and Dunkl processes with drift}

\author{ Michael Voit} 
\address{Fakult\"at Mathematik, Technische Universit\"at Dortmund,
          Vogelpothsweg 87,
          D-44221 Dortmund, Germany}
\email{ michael.voit@math.tu-dortmund.de}


\begin{document}
\date{\today}

\begin{abstract} For some discrete parameters, multivariate (Dunkl-)Bessel processes on Weyl chambers $C$
 are  projections of Brownian motions  on Euclidean spaces $V$;
  the most prominent examples are Dyson Brownian motions. More generally,
  the projections of Brownian motions on $V$ with drift are Bessel processes with drifts on $C$, where
 again  the associated transition densities  can be described in terms of multivariate Bessel functions.
  Moreover,  Dunkl processes with drift can be defined in an analogous  way.

  In this paper we prove  central limit theorems for these Bessel and Dunkl processes with drift for
  regular drift vectors, and arbitrary root systems and multiplicities.
  These results improve recent results of the author.
\end{abstract}

\maketitle

\section{Introduction}

In this paper we prove a central limit theorem (CLT) for multivariate Dunkl and Bessel processes
with regular drift vectors which were partially conjectured in  \cite{V}. The proof will be based on some
limit result for Dunkl kernels which is based on some estimates for these kernels in \cite{R3}.
In the one-dimensional case, a version of the CLT was proved in  \cite{V}.

To explain  Dunkl and Bessel processes without and with drift 
as well as the associated special functions, i.e.,  Dunkl kernels and Bessel functions associated with root systems, we recapitulate some notations
from
\cite{D, R2, R4, RV1, RV3, CGY, GY1, GY2,  Ka, A}.
Let $\mathbb R^N$ ($N\ge1$) be
equipped with  the usual scalar product   $\langle\,.\,,\,.\,\rangle$, the associated 2-norm $\|.\|$,
and the standard basis $e_1,\ldots,e_N$.
Let $R\subset \mathbb R^N\setminus\{0\}$ be a  root system on $\mathbb R^N$
i.e., $R$ is a finite set of vectors $\alpha$, such that the reflections
 $\sigma_\alpha$  on the  hyperplanes perpendicular to $\alpha$  
generate a finite group $W$ of orthogonal transformations on  $\mathbb R^N$ with $\sigma_\alpha(R)=R$
for   $\alpha\in R$. 
Fix some positive subsystem $R_+\subset R$, the associated 
 closed Weyl chamber
 $$C:=\{ x\in \mathbb R^N: \quad \langle x, \alpha\rangle\ge 0 \quad\text{for all}\quad \alpha\in R_+\},$$ and
 some multiplicity function $k:R\to[0,\infty[$, i.e.,
  $k$ is invariant under the action of $W$ on $R$. For instance, for the root systems
  $A_{N-1}$ and $D_N$ on  $\mathbb R^N$, this $k$ is  a single parameter $k\in [0,\infty[$, and, for
 the root system $B_N$ ($N\ge 2$),   $k$ consists of 2 parameters 
  $(k_1,k_2)$. The chamber $C$     
         consists of representatives of the orbits under the action of $W$ on  $\mathbb R^N$.
We now consider
the associated Dunkl operators 
\begin{equation}\label{def-dunkl-op}
  T_\xi(k) f(x)   = \partial_{\xi}  f(x)+\sum_{\alpha\in R_+}
  k(\alpha)\cdot \langle\alpha,\xi\rangle \frac{1}{\langle\alpha, x\rangle} (f(x)-f(\sigma_\alpha(x)))
\end{equation} 
 for  $x,\xi \in \mathbb R^N$ and the directional derivatives $\partial_{\xi}$.
The  $T_\xi(k)$ form a  commuting family of  operators which are homogeneous of degree $-1$ on 
the space 
of all polynomials in $N$ variables such that, like for the classical directional derivatives for $k=0$,
we have $T_\xi(k) f\in {\cal C}^{(r-1)}(\mathbb R^N)$ for all 
$f\in  {\cal C}^{(r)}(\mathbb R^N)$ and $r=1,2,\ldots$.
Consider the associated Dunkl kernels $E_k(z,w)$ ($z,w\in \mathbb C^N$) where the functions $E(.,w)$ are characterized as the
unique holomorphic solutions of the joint eigenvalue problem 
\begin{equation}\label{joint-Dunkl-eigenvalues}
  T_\xi(k) f\,=\, \langle\xi,w\rangle f \quad (\xi\in \b C^N) \quad\text{with}\quad  \quad f(0)=1.
\end{equation}
The $E_k$ have the following properties for  $x,y\in \mathbb C^N$:
\begin{enumerate}\itemsep=-1pt
\item[\rm{(1)}] $E_k(x,y) = E_k(y,x).$ 
\item[\rm{(2)}]  $E_k(\lambda x,y) = E_k(x,\lambda y)$ for $ \lambda\in \mathbb C$.
\item[\rm{(3)}]  $0<E_k( x,y) \le \max_{g\in W} e^{\langle x, gy\rangle}$ and  $|E_k( x,iy)|\le 1$ for $x,y\in \mathbb R^N$.
 \item[\rm{(4)}]  $E_k(gx,gy) = E_k(w,y)$ for $g\in W$.
\end{enumerate}
The  Bessel functions associated the Dunkl kernels are given  as the symmetrizations
$$ J_k(z,w) := \frac{1}{|W|}\sum_{g\in W} E_k(z,gw)  \quad\quad (z,w\in\mathbb C^N).$$
We next consider the  Dunkl-Laplacian $\Delta_k:=\sum_{j=1}^N T_{e_j}^2$ which is explicitly given by
\begin{equation}\label{Dunkl-laplace}
 \Delta_k f(x) = \Delta f(x)  + \,2\sum_{\alpha\in R_+}
k(\alpha)\Bigl(
\frac{\langle\nabla f(x),\alpha\rangle}{\langle\alpha ,x\rangle} -\frac{ \|\alpha\|^2}{2}
\frac{f(x)-f(\sigma_\alpha x)}{\langle\alpha , x\rangle^2}\Bigr).
\end{equation}
 $\frac{1}{2}\Delta_k$ is the generator of a Feller process
$(X_{t,k})_{t\ge0}$ on $\mathbb R^N$ with RCLL paths which generalizes Brownian motions, which appear for $k=0$.
 These processes have
 the transition  probabilities
\begin{equation}\label{density-transition-dunkl-gen}
K_t^k(x,A)=c_k \int_A \frac{1}{t^{\gamma+N/2}} e^{-(\|x\|^2+\|y\|^2)/(2t)} E_k(\frac{x}{\sqrt{t}}, \frac{y}{\sqrt{t}}) 
\cdot w_k(y)\> dy
\end{equation}
for $t>0$, $x\in \mathbb R^N$, and Borel sets $A\subset \mathbb R^N$
with  the  $W$-invariant  weight function
\begin{equation}\label{def-wk-A-general}
   w_k(x) := \prod_{\alpha\in R_+} |\langle\alpha,x\rangle|^{2k(\alpha)}.
\end{equation}
 which is homogeneous of  degree $2\gamma$ with 
$\gamma:=\sum_{\alpha\in R_+}k(\alpha)$, where
 $c_k>0$ is a explicitly known  normalization; see for instance   \cite{M} for these constants
 for the root systems $A_{N-1}$ and $B_N$.
The processes $(X_{t,k})_{t\ge0}$
are called  Dunkl processes associated with $R$ and $k$ (without drift).

 Furthermore,  the operator
 \begin{equation}\label{Bessel-laplace}
  L_k f(x):= \frac{1}{2}\Delta_k f(x)=\frac{1}{2}\Delta f(x)  + \sum_{\alpha\in R_+}
k(\alpha)
\frac{\langle\nabla f(x),\alpha\rangle}{\langle\alpha ,x\rangle}
\end{equation}
for $W$-invariant $f\in C^{(2)}(\mathbb R^N)$ is the generator of a Feller diffusion
$(X_{t,k})_{t\ge0}$ on the closed chamber
$C$ with continuous paths with reflecting boundaries with the transition  probabilities 
\begin{equation}\label{density-transition-bessel-gen}
K_t^k(x,A)=|W|c_k \int_A \frac{1}{t^{\gamma+N/2}} e^{-(\|x\|^2+\|y\|^2)/(2t)} J_k(\frac{x}{\sqrt{t}}, \frac{y}{\sqrt{t}}) 
\cdot w_k(y)\> dy
\end{equation}
for $t>0$, $x\in C$, and $A\subset C$. These diffusions are called Bessel processes on $C$ (without drift).

These Dunkl and Bessel processes without drift were extended to processes with drift in a systematic way in \cite{V}.
This construction is motivated by examples of such Bessel processes, which appear as projections of
Brownian motions with drift on special Euclidean spaces like the spaces of complex $N\times N$ Hermitian matrices as in \cite{AV},
and similar constructions for 1-dimensional Bessel processes in \cite{PY}, and for Wishart processes in \cite{DDMY}.
The Dunkl and Bessel processes with drift are defined as follows for $R,R_+, W, C, k\ge0$ as above:

\begin{definition}\label{def-bessel-p-drift}
 Let $\lambda\in C$ be a ``drift vector''. Then the operator
   \begin{equation}\label{Bessel-laplace-drift-gen}
  L_k^\lambda f(x):= \frac{1}{2}\Delta f(x)  + \sum_{\alpha\in R_+}
k(\alpha)
\frac{\langle\nabla f(x),\alpha\rangle}{\langle\alpha ,x\rangle}+\frac{\langle \nabla_x J_k(x,\lambda), \nabla f(x)\rangle }{J_k(x,\lambda)}
   \end{equation}
   for $W$-invariant $f\in C^{(2)}(\mathbb R^N)$ is the generator of a Feller diffusion on $C$ with reflecting boundaries and the
  transition densities
   \begin{equation}\label{density-transition-bessel-drift-gen}
  K_t^{\lambda,k}(x,A)=\frac{|W|\cdot c_k\cdot e^{-\|\lambda\|^2 t/2}}{t^{\gamma+N/2}} \int_A e^{-(\|x\|^2+\|y\|^2)/(2t)}
 \frac{ J_k(\frac{x}{\sqrt{t}}, \frac{y}{\sqrt{t}}) \cdot J_k(y,\lambda)}{J_k(x,\lambda)}
\cdot w_k(y)\> dy
\end{equation}
for $t>0$, $x\in C$ and  Borel sets $A\subset  C$.
These diffusions  have modifications with continuous paths; they
are  called Bessel processes of type $R$ with multiplicity $k$ and drift $\lambda$.   
\end{definition}

 \begin{definition}\label{def-dunkl--p-drift}
     Let $\lambda\in C$ be a ``drift vector''.
   The operator
   \begin{align}\label{Dunkl-laplace-drift-gen}
  L^{\lambda, Dunkl}_k f(x):= &\frac{1}{2}\Delta f(x)  + \sum_{\alpha\in R_+}
k(\alpha)
\frac{\langle\nabla f(x),\alpha\rangle}{\langle\alpha ,x\rangle} + 
\frac{\langle \nabla_x E_k(x,\lambda), \nabla f(x)\rangle }{E_k(x,\lambda)}\\
&-\sum_{\alpha\in R_+}\frac{k(\alpha)\|\alpha\|^2}{2}\cdot \frac{ E_k(\sigma_\alpha x,\lambda)}{ E_k(x,\lambda)}\cdot
\frac{f(x)-f(\sigma_\alpha x)}{\langle\alpha , x\rangle^2} \notag
   \end{align}
   for  $f\in C^{(2)}(\mathbb R^N)$ is the generator of a Feller process on $\mathbb R^N$ with
   the transition densities
   \begin{equation}\label{density-transition-dunkl-drift-gen}
  K_t^{\lambda,k,Dunkl}(x,A)=\frac{ c_k\cdot e^{-\|\lambda\|^2 t/2}}{t^{\gamma+N/2}} \int_A e^{-(\|x\|^2+\|y\|^2)/(2t)}
 \frac{ E_k(\frac{x}{\sqrt{t}}, \frac{y}{\sqrt{t}}) \cdot E_k(y,\lambda)}{E_k(x,\lambda)}
\cdot w_k(y)\> dy
\end{equation}
for $t>0$, $x\in \mathbb R^N$ and $A\subset \mathbb R^N$.
These Feller processes have RCLL modifications.
Such processes are called Dunkl processes of type $R$ with multiplicity $k$ and drift $\lambda$.
 \end{definition}

 By \cite{V}, Dunkl processes $(X_t^\lambda)_{t\ge0}$ with arbitrary drift vectors $\lambda\in\mathbb R^N$ and start in 0
 have the property that $(X_t^\lambda-t\lambda)_{t\ge0}$ are  martingales w.r.t.~the canonical filtration, a fact which supports the
 name of these processes. Moreover, some strong law of large numbers holds for these proceses; see \cite{V}.
 The purpose of this note is to prove the following associated CLTs for Dunkl and Bessel processes with regular drift $\lambda$,
 i.e.,  $\lambda$ is not contained in the boundary of some Weyl chamber.

\begin{theorem}\label{clt-dunkl}
  For any Dunkl process  $(X_t^\lambda)_{t\ge0}$ with $k\ge0$, 
regular    drift  vector $\lambda$,  and any starting point $x\in\mathbb R^N$,
  the random variables    $(X_t^\lambda-t\lambda)/\sqrt t$ tend in distribution
to the normal distribution $N(0,I_N)$ for $t\to\infty$.
 \end{theorem}

This CLT can be extended to starting points depending on $t$ under the following  restriction:

\begin{definition}\label{def-regular-pairs}
  Pairs $(x,\lambda)\in \mathbb R^N\times \mathbb R^N$ of points are called regular, if $\lambda$ and $x$ are regular and
   in the interior of the same Weyl chamber
  $C$, and if 
  \begin{equation}\label{est-def-regular-pairs}
    \|x-\lambda\|< C(x,\lambda):=\text{max}(dist(x,\partial C), dist(\lambda,\partial C)\}
    \end{equation}
where $dist$ is the usual Euclidean distance of some point and some set in  $\mathbb R^N$.
  \end{definition}

With this notation we have the following  CLT.

 \begin{theorem}\label{clt-dunkl-general}
   Let  $(X_t^{\lambda, x})_{t\ge0}$ be a Dunkl process
with $k\ge0$, 
drift  vector $\lambda$,  and start in $x\in\mathbb R^N$. If  $(\lambda,x)\in\mathbb R^N\times\mathbb R^N$
is regular, 
then the random variables
     $(X_t^{\lambda, tx}-t(\lambda+x))/\sqrt t$ tend in distribution
to  $N(0,I_N)$ for $t\to\infty$.
 \end{theorem}

 Clearly, for $x=0$,  the assertions of the CLTs \ref{clt-dunkl} and \ref{clt-dunkl-general}  are equal where this case formally is not covered
 by  Theorem \ref{clt-dunkl-general}. For this reason we conjecture that
 Theorem \ref{clt-dunkl-general} still holds under slightly  more general assumptions.
On the other hand,
we  show in the end of Section 2  that for $N=1$  the assertion of Theorem \ref{clt-dunkl-general} does not hold for $k>0$ and all
 $\lambda,x$ in different Weyl chambers.

 We shall prove Theorems \ref{clt-dunkl} and \ref{clt-dunkl-general}
 with the following limit result for the Dunkl kernels which is a
 variant of Theorem 2 in \cite{RdJ}, and which generalizes Corollary 3.4 in \cite{R3} slightly:

 \begin{proposition}\label{limit-dunkl-kernel} Let  $(\lambda,x)\in\mathbb R^N\times\mathbb R^N$ such that
there are $\tau_1,\tau_2>0$ such that  $(\tau_1\lambda,\tau_2 x)$ is regular.
 Let $\phi:[0,\infty[\to[0,\infty[$ be a decreasing function with $\lim_{t\to\infty} \phi(t)=0$. Then,
           locally uniformly for $y\in\mathbb R^N$,
 \begin{equation}\label{dunkl-kernel-limit}
  \lim_{t\to\infty} E_k(t\lambda +t\> \phi(t)y, x)
  \cdot (2\pi)^{N/2} c_k e^{\langle-t \lambda - t\> \phi(t)y, x\rangle} \sqrt{w_k(t\lambda+ t\> \phi(t)y)w_k(x)}=1.
 \end{equation}
\end{proposition}

We shall also prove the following CLT  for Bessel processes with drift:

\begin{theorem}\label{clt-bessel}
  For any Bessel process  $(X_t^\lambda)_{t\ge0}$ on some Weyl chamber $C$ with $k\ge0$, 
regular    drift  vector $\lambda\in C$,  and start in any $x\in C$,
 the random variables    $(X_t^\lambda-t\lambda)/\sqrt t$ tend in distribution
to the standard normal distribution $N(0,I_N)$ for $t\to\infty$.
\end{theorem}

Theorems \ref{clt-dunkl} and \ref{clt-bessel} clearly lead to the following law of large numbers::

\begin{corollary}\label{cor-wslln} Let $(X_t^\lambda)_{t\ge0}$ be a Dunkl or Bessel Process with 
  regular drift $\lambda$ and arbitrary
  starting point $x$ (with we assume $x,\lambda\in C$ for a Bessel process on the chamber $C$).
  Then $X_t^\lambda/t\to \lambda$ for $t\to\infty$ in probability.
\end{corollary}

This result should be compared with the strong law of large numbers $m_1(X_t^\lambda)/t\to \lambda$
for $t\to\infty$ a.s for regular $\lambda$
in
\cite{V} with certain functions $m_1:\mathbb R^N\to\mathbb R^N$ depending on $\lambda$ where for  $y\in\mathbb R^N$ with $\|y\|$
large,
$m_1(y)$ is approximately equal to the projection of $y$ to the Weyl chamber $C$ which contains $\lambda$.
Unfortunately we are are not able to combine these strong laws from \cite{V} with Corollary \ref{cor-wslln} to conclude that
$X_t^\lambda/t\to \lambda$  a.s.~for $t\to\infty$ under the conditions of Corollary \ref{cor-wslln}.
However, we conjecture thta this strong law holds.

We finally notice that in the CLT  \ref{clt-bessel} the regularity of the drift $\lambda$ is necessary for $k>0$,
as for drift vectors $\lambda$ on the boundary of $C$
   the supports of the distributions of $X_t^\lambda-t\lambda)/\sqrt t$
   are contained in some half space, i.e., the limits cannot be normal. We conjecture that here the limit distributions
   are multivariate $\chi^2$-distributions (after taking squares in all coordinates).

Please notice that for $k=0$, Dunkl processes are just Brownian motions with drift $\lambda$, i.e., 
here the CLT \ref{clt-dunkl-general} holds for all $\lambda,x\in\mathbb R^N$ in the Dunkl case,
and the limit distributions for the associated Bessesl processes are certain projections of normal distributions. 

We prove the CLTs \ref{clt-dunkl}, \ref{clt-dunkl-general}, \ref{clt-bessel} as well as  Proposition \ref{limit-dunkl-kernel} 
in the following section.

\section{Proofs of the main results}

We first prove Proposition \ref{limit-dunkl-kernel}. We  fix some root system on $\mathbb R^N$ and some multiplicity $k\ge0$.
We consider the renormalized Dunkl heat kernels
$$\tilde\Gamma_k(t,x,y):=\frac{c_k\sqrt{w_k(x)\cdot w_k(y)}}{t^{\gamma+N/2} }
 e^{-(\|x\|^2+\|y\|^2)/(2t)} E_k(\frac{x}{\sqrt{t}}, \frac{y}{\sqrt{t}})  \quad\quad(t>0, \>  x,y\in\mathbb R^N).$$
 We also need the following notations from \cite{R3}. Let $C\subset \mathbb R^N$ be some Weyl chamber.
 For $\delta>0$ we then define
 $$C_\delta:= \{x\in C:\> \|x\|<1/\delta; \> dist(x,\partial C)>\delta\} \quad \text{and}$$
 $$C_\delta(x,\lambda):=\text{max}(dist(x,\partial C_\delta), dist(\lambda,\partial C_\delta)\} \quad\quad(x,\lambda\in C)$$
similar to Definition \ref{def-regular-pairs}.
We then have $C_{\tilde \delta}\subset C_\delta\subset C$ for $\tilde\delta>\delta$. Moreover, for regular $(x,\lambda)$, 
$$C_\delta(x,\lambda)\le C(x,\lambda) \quad\text{with} \quad C_\delta(x,\lambda)\to C(x,\lambda) \quad\text{for}  \quad\delta\to0.$$
Furthermore, the following estimate is shown in Lemma 4.4 of  \cite{R3}:

\begin{lemma}\label{roesler-estimate}
  Let $\delta>0$. Then there exist constants $M_1,M_2>0$ such that for all $x,y\in C_\delta$ and $t>0$,
   $$\tilde\Gamma_0(t,x,y)-e^{-tM_1}\tilde\Gamma_k(t,x,y)\le
  \frac{c_0\cdot 2^{N/2}}{t^{N/2}} e^{-C_\delta(x,y)^2/2t}, $$
  and
 $$e^{-tM_2}\tilde\Gamma_k(t,x,y)-\tilde\Gamma_0(t,x,y)\le \frac{c_k\cdot 2^{N/2}}{t^{\gamma+N/2} \delta^{2\gamma}}
       e^{-C_\delta(x,y)^2/2t}.$$
\end{lemma}

This result leads to the following short-time asymptotics for the  renormalized Dunkl heat kernels:

\begin{lemma}\label{short-time}
  Let $\phi:[0,\infty[\to[0,\infty[$ be a decreasing function with $\lim_{t\downarrow 0} \phi(t)=0$.
          Let $(\lambda,x)\in \mathbb R^N\times\mathbb R^N $ be regular. Then, locally uniformly in $y\in\mathbb R^N$,
      $$\lim_{t\downarrow 0} \frac{\tilde\Gamma_k(t,x, \lambda+\phi(t)y)}{\tilde\Gamma_0(t,x, \lambda+\phi(t)y)}=1.$$
\end{lemma}

\begin{proof} Let $C$ be the Weyl chamber with $\lambda$ and $x$ is in its interior. Fix some compactum $K\subset\mathbb R^N$.
  We now choose some $\delta>0$ and $t_0>0$ with $x\in C_\delta$ and $\lambda+\phi(t)y\in C_\delta$ for all $t\in[0,t_0]$ and $y\in K$,
  where we may also assume that
  \begin{equation}\label{est-strong}
      \|x-\lambda\|-C_\delta(x,\lambda)<0.\end{equation}
We then conclude from Lemma
  \ref{roesler-estimate} that for these $t,y$,
$$1-e^{-tM_1}\frac{\tilde\Gamma_k(t,x, \lambda+\phi(t)y)}{\tilde\Gamma_0(t,x, \lambda+\phi(t)y)}
  \le  e^{(\|x-\lambda-\phi(t)y\|^2 -C_\delta(x, \lambda+\phi(t)y)^2)/4t}$$
  and
  $$e^{-tM_2}\cdot \frac{\tilde\Gamma_k(t,x, \lambda+\phi(t)y)}{\tilde\Gamma_0(t,x, \lambda+\phi(t)y)}-1\le
  \frac{c_k}{\delta^{2\gamma} c_0 t^\gamma}  e^{(\|x-\lambda-\phi(t)y\|^2-C_\delta(\lambda, \lambda+\phi(t)y)^2)/4t}.$$
As by \eqref{est-strong}
$$\|x-\lambda-\phi(t)y\|^2 -C_\delta(x, \lambda+\phi(t)y)^2$$
tends to a negative constant for $t\downarrow 0$, both bounds on he right hand sides above tend to
  $0$ for $t\downarrow 0$   uniformly in $y\in K$. Hence, these estimates lead to the claim.
\end{proof}

\begin{proof}[Proof of Proposition \ref{limit-dunkl-kernel}]
If $(x,\lambda)$ is regular, then this result follows immediately from 
Lemma \ref{short-time} with  $1/t$ instead of $t$ and 
the definition of the kernels $\tilde\Gamma_k$ above. The general case then follows by a renormalization of $t$ and property (2)
of the Dunkl kernel in Section 1.
\end{proof}

We next turn to the proof of Theorem \ref{clt-dunkl-general}.  We here need the following  fact:

\begin{lemma}\label{mean-prop} If $(\lambda,x)$ is regular, then $((x+\lambda)/2,x)$ and $((x+\lambda)/2,\lambda)$  are regular.
\end{lemma}

\begin{proof} Clearly, $x+\lambda$ is also contained in the interior of the Weyl chamber $C$ which contains $x,\lambda$.
  Assume now w.l.o.g.~that $dist(x,\partial C)\le  dist(\lambda,\partial C)$.
  Then
  $$\|(x+\lambda)/2-\lambda\| = \|x-\lambda\|/2 <  dist(\lambda,\partial C)$$
  which shows that  $((x+\lambda)/2,\lambda)$ is regular. Moreover, as $C$ is an intersection of halfspaces,
  $$\|(x+\lambda)/2-x\|  = \|x-\lambda\|/2 < \bigl( dist(x,\partial C)+ dist(\lambda,\partial C)\bigr)/2 = dist((x+\lambda)/2,\partial C),$$
  which shows that $((x+\lambda)/2,x)$ is regular.
\end{proof}

\begin{proof}[Proof of Theorem \ref{clt-dunkl-general}]
 Let  $(X_t^{\lambda, x})_{t\ge0}$ be a Dunkl process
with multiplicity $k\ge0$,     drift  $\lambda$,  and start in $x\in\mathbb R^N$ such that $(\lambda,x)$ is regular.
Using the transition probabilities \eqref{density-transition-dunkl-drift-gen}, we see that
for   $t>0$, the random variable
     $(X_t^{\lambda, tx}-t(\lambda+x))/\sqrt t$  has the Lebesgue-density
\begin{align}\label{density-transition-dunkl-drift-gen-mod}    
 & \frac{ c_k\cdot e^{-\|\lambda\|^2 t/2}}{t^{\gamma}}  e^{-(t^2\|x\|^2+\| t(\lambda+x)+  \sqrt t\cdot y\|^2)/(2t)}
 E_k\Bigl(\frac{tx}{\sqrt{t}}, \frac{    t(\lambda+x)+  \sqrt t\cdot y}{\sqrt{t}}\Bigr) \\
 &\quad \cdot \frac{ E_k( t(\lambda+x)+  \sqrt t\cdot y,\lambda) }{E_k(tx,\lambda)}
 \cdot w_k(t(\lambda+x)+  \sqrt t\cdot y ) \notag\\
 =& c_k e^{-t(\|\lambda\|^2+\|x\|^2+\|\lambda+x\|^2)/2} 
 \cdot \frac{E_k(x,  t(\lambda+x)+  \sqrt t\cdot y) E_k(\lambda,  t(\lambda+x)+  \sqrt t\cdot y) }{E_k(tx,\lambda)}
 \notag\\
 &\quad \cdot
 w_k( \sqrt t(\lambda+x)+  y)\cdot
e^{- \sqrt t\langle \lambda+x,y\rangle}  e^{-\|y\|^2/2}
\notag\end{align}
where the factor $t^{N/2}$ disappeared by the transformation formula, and where we used property (2) of $E_k$ in Section 1 and the homogeneity of $w_k$ for the ``$=$'' in \eqref{density-transition-dunkl-drift-gen-mod}. Using Lemma \ref{mean-prop}, we can now apply
Proposition \ref{limit-dunkl-kernel} to the three $E_k$  on the right hand side of 
\eqref{density-transition-dunkl-drift-gen-mod} with the function $\phi(t)=t^{-1/2}$ and  $\phi(t)=0$ respectively.
This and an elementary calculation imply that for $t\to\infty$, these densities converge locally uniformly in $y\in\mathbb R^N$ to
the density $(2\pi)^{-N/2} e^{-\|y\|^2/2}$ of $N(0,I_N)$. 
This means that  the distributions of $(X_t^{\lambda, tx}-t(\lambda+x))/\sqrt t$  tend vaguely and thus weakly to $N(0,I_N)$ as claimed.
\end{proof}

\begin{proof}[Proof of Theorem \ref{clt-dunkl}]
  Similar to \eqref{density-transition-dunkl-drift-gen-mod} in the preceding proof we here see that
  $(X_t^{\lambda}-t\lambda)/\sqrt t$ has the density
\begin{equation}\label{density-transition-dunkl-drift-gen-mod2}    
  c_k e^{-t\|\lambda\|^2} e^{-\|x\|^2/(2t)}e^{-\|y\|^2/2}e^{-\sqrt t \langle \lambda,y\rangle}
w_k( \sqrt t \cdot \lambda+  y) 
 \cdot \frac{E_k(x,  \lambda+  y/ \sqrt t) E_k(\lambda,  t\lambda+  \sqrt t\cdot y) }{E_k(x,\lambda)}.
 \end{equation}
We now apply Proposition \ref{limit-dunkl-kernel} to $E_k(\lambda,  t\lambda+  \sqrt t\cdot y) $ and use that
$$\frac{E_k(x,  \lambda+  y/ \sqrt t)}{E_k(x,\lambda)}\to1$$
for $t\to\infty$ locally uniformly in $y$. This implies again that
the distributions of $(X_t^{\lambda, tx}-t\lambda)/\sqrt t$  tend vaguely and thus weakly to $N(0,I_N)$..
\end{proof}

In order to derive the CLT \ref{clt-bessel} from that for Dunkl processes with drift,
we use the following fact on the space  $M^1(\mathbb R^N)$  of all
probability measures on $\mathbb R^N$.

\begin{lemma}\label{weak-convergence-lemma}
  Let $(\mu_t)_{t\ge0}\subset M^1(\mathbb R^N)$ be a familiy of probability measures of the form
  $\mu_t=\tilde\mu_t|_{C_t}+\hat\mu_t$ with  $(\tilde\mu_t)_{t\ge0}\subset M^1(\mathbb R^N)$
  converging weakly to some $\mu\in M^1(\mathbb R^N)$ for $t\to\infty$, $(C_t)_{t \ge0}$ a family of open subsets on
  $\mathbb R^N$ with $C_s\subset C_t$ for $s<t$ and $\bigcup_{t\ge0} C_t=\mathbb R^N$,
  and with (nonnegative) sub-probability measures $\hat\mu_t:=\mu_t-\tilde\mu_t|_{C_t}$ ($t\ge0$).
  Then  $(\mu_t)_{t\ge0}$ also converges weakly to  $\mu$ for $t\to\infty$.
  \end{lemma}

\begin{proof}
  We first check that $\lim_{t\to\infty}\tilde\mu_t(C_t)=1$. In fact, for each $\epsilon>0$,
  we have $\mu(C_t)\ge1-\epsilon$ for $t\ge t_0$ with some $t_0=t_0(\epsilon)$.
  We next choose some continuous function
  $f:\mathbb R^N\to[0,1]$ with $supp\> f\in C_{t_0}$ with $\int_{\mathbb R^NN}f\> d\mu\ge 1-2\epsilon$.
  Hence, by our assumptions, we find $t_1\ge t_0$ such that for $t\ge t_1$ we have
  $\int_{\mathbb R^N}f\> d\tilde\mu_t\ge 1-3\epsilon$ and thus
  $$\tilde\mu_t(C_t)\ge\tilde\mu_t(C_{T_0})\ge 1-3\epsilon.$$
  Hence,  $\lim_{t\to\infty}\tilde\mu_t(C_t)=1$ and thus  $\lim_{t\to\infty}\hat\mu_t(C_t)=0$.

  Now let $f:\mathbb R^N\to \mathbb R$ be a  continuous function with compact support.
  Then $supp \>f \subset C_T$
  for large $t$ by a compactness argument. We thus conclude that
  $$\lim_{t\to\infty}\int_{\mathbb R^N}f\> d\mu_t=\lim_{t\to\infty}\int_{\mathbb R^N}f\> d\tilde\mu_t
  + \lim_{t\to\infty}\int_{\mathbb R^N}f\> d\hat\mu_t=\int_{\mathbb R^N}f\>d\mu +0.$$
  Therefore, the familily  $(\mu_t)_{t\ge0}$ of probability measures converges vaguely and thus weakly
  to $\mu$ as claimed.
  \end{proof}

  \begin{proof}[Proof of Theorem \ref{clt-bessel}]
    Let $\lambda$ in the interior of the chamber $C$. Let $(X_t^{\lambda,D})_{t\ge0}$ be a
    Dunkl process on $\mathbb R^N$ with start in $0$ and  $(X_t^{\lambda,B})_{t\ge0}$  a Bessel process on $C$
with start in $x\in C$, where both processes have drift $\lambda$.
As in the proof of Theorems \ref{clt-dunkl-general} and \ref{clt-dunkl}, we conclude from \eqref{density-transition-dunkl-drift-gen}
and \eqref{density-transition-bessel-drift-gen}  that the random variables
$(X_t^{\lambda, D}-t\lambda)/\sqrt t$  and $(X_t^{\lambda, B}-t\lambda)/\sqrt t$ have the distributions
\begin{equation}
d\tilde\mu_t(y):= H_{t,\lambda}(y) \cdot
E_k(\lambda,t\lambda+\sqrt t\cdot y)\> dy \in M^1(\mathbb R^N)\end{equation}
and
\begin{equation}d\mu_t(y):= |W| \cdot H_{t,\lambda}(y)\cdot\frac{J_k(x,  \lambda+  y/ \sqrt t)}{J_k(x,\lambda)}
  J_k(\lambda,t\lambda+\sqrt t\cdot y)\> dy|_{C_t} \in M^1(C_t)
\end{equation}
respectively with the translated cones
$C_t:=C-\sqrt t\cdot\lambda$ and with
$$H_{t,\lambda}(y)= c_k\cdot e^{-\|\lambda\|^2 t/2}\cdot  e^{-\|t\lambda +\sqrt t\cdot y\|^2/(2t)} w_k(\sqrt t\cdot \lambda+y).$$
We also consider the probability measures
\begin{align}
  d\check\mu_t(y):=& |W| \cdot H_{t,\lambda}(y)\cdot
  J_k(\lambda,t\lambda+\sqrt t\cdot y)\> dy|_{C_t} \in M^1(C_t)
  \\
  =& {\bf 1}_{C_t}\cdot H_{t,\lambda}(y)\cdot
  E_k(\lambda,t\lambda+\sqrt t\cdot y)\> dy \notag\\
  &\quad +{\bf 1}_{C_t}\cdot H_{t,\lambda}(y)\cdot \sum_{g\in W, g\ne e} E_k(\lambda,t\lambda+\sqrt t\cdot y)\> dy \notag\\
  =&d\tilde\mu_t(y)|_{C_t} +d\hat\mu_t(y)\notag
\end{align}
 with some positve subprobability $\hat\mu_t$. As the cones $C_t$ and
  the probability measures $\tilde\mu_t$ satisfy the conditions
  of Lemma \ref{weak-convergence-lemma} by the CLT \ref{clt-dunkl} with $\mu=N(0,I_N)$ as limit,
  we obtain that the $\check\mu_t$ also tend weakly to $N(0,I_N)$. As
  $$\frac{J_k(x,  \lambda+  y/ \sqrt t)}{J_k(x,\lambda)}\to 1$$
  locally uniformly in $y$, it follows easily that  the probability measures $\mu_t$ also tend 
 weakly to $N(0,I_N)$ as claimed.
  \end{proof}

  We expect that  the CLT \ref{clt-dunkl-general} for Dunkl processes can be also transfered to Bessel processes.
  As here means of the Dunkl kernels $E_k$ over the Weyl group $W$ have to be taken also in the denominators
  in  the transformed densities \eqref{density-transition-bessel-drift-gen}, this is more involved than in the preceding proof.

\begin{example} In the end of this paper
  we show that for $N=1$, $k>0$, drift $\lambda>0$ and starting point $x<0$, the CLT \ref{clt-dunkl-general} is
not correct. The same would also hold for $\lambda<0$ and  $x>0$. In fact,
it is known (see e.g. \cite{R1}) that here
  $$E_k(x,\lambda)= e^{x \lambda}\>_1F_1(k,2k+1; -2x \lambda).$$
The Kummer transformation for $_1F_1$ together with a   limit relation for 
  $_1F_1(a,b;z)$ for $z\to\infty$ (see 13.2.39 and  13.2.23 of \cite{NIST})
now yield
  \begin{align}\label{limit-1f1}
    E_k(x,\lambda)\sim& \frac{\Gamma(2k+1)}{\Gamma(k+1)2^k} e^{x \lambda}\cdot (x\lambda)^{-k} \quad\quad\text{for}\quad\quad x\to\infty
 \quad\quad\text{and}   \notag\\
 E_k(x,\lambda)\sim& \frac{\Gamma(2k+1)}{\Gamma(k)2^{k+1}} e^{-x \lambda}\cdot (-x\lambda)^{-(k+1)}
 \quad\quad\text{for}\quad\quad x\to-\infty.\end{align}
These results were also used in  \cite{V}. 
This implies readily that for  $\lambda>0$ and starting point $x<0$, the limit $t\to\infty$ of the densities
\eqref{density-transition-dunkl-drift-gen-mod} does not exist in particular due to the fact that the exponential terms
do not cancel anymore mainly due to the different exponential terms on the right hand sides of \eqref{limit-1f1} for $x>0$ and $x<0$.
One might suspect that this failure can be repaired by considering slightly modified limits for the random variables  like 
$(X_t^{\lambda, tx}-st)/(c\sqrt t)$ for  suitable $s=s(x,\lambda)$, $c=c(x,\lambda)>0$. Unfortunately, this approach still does not repair
the  additional problems with the powers $k$ and $k+1$ on the right hand sides of \eqref{limit-1f1}.
\end{example}

\end{document}